\documentclass{amsart}
\usepackage{amssymb,
enumitem,
mathrsfs,
mathtools,
tikz,
upgreek,
verbatim,
hyphenat
}
\usepackage[T1]{fontenc}
\usepackage[shadow]{todonotes}
\usepackage[hidelinks]{hyperref}

\newcommand{\NN}{\mathbb{N}}
\newcommand{\ZZ}{\mathbb{Z}}

\newcommand{\RR}{\mathbb{R}}

\newcommand{\setm}[2]{\{ #1 : #2 \}}

\renewcommand{\wr}{\mathbin{\mathrm{wr}}}

\DeclareMathOperator{\LO}{\mathrm{LO}}

\newenvironment{enumerate-(a)}{\begin{enumerate}[label={\upshape (\alph*)}, leftmargin=2pc]}{\end{enumerate}}
\newenvironment{enumerate-(a)-r}{\begin{enumerate}[label={\upshape (\alph*)}, leftmargin=2pc,resume]}{\end{enumerate}}
\newenvironment{enumerate-(a)-5}{\begin{enumerate}[label={\upshape (\alph*)}, leftmargin=2pc,start=5]}{\end{enumerate}}
\newenvironment{enumerate-(A)}{\begin{enumerate}[label={\upshape (\Alph*)}, leftmargin=2pc]}{\end{enumerate}}
\newenvironment{enumerate-(A)-r}{\begin{enumerate}[label={\upshape (\Alph*)}, leftmargin=2pc,resume]}{\end{enumerate}}
\newenvironment{enumerate-(i)}{\begin{enumerate}[label={\upshape (\roman*)}, leftmargin=2pc]}{\end{enumerate}}
\newenvironment{enumerate-(i)-r}{\begin{enumerate}[label={\upshape (\roman*)}, leftmargin=2pc,resume]}{\end{enumerate}}
\newenvironment{enumerate-(I)}{\begin{enumerate}[label={\upshape (\Roman*)}, leftmargin=2pc]}{\end{enumerate}}
\newenvironment{enumerate-(I)-r}{\begin{enumerate}[label={\upshape (\Roman*)}, leftmargin=2pc,resume]}{\end{enumerate}}
\newenvironment{enumerate-(1)}{\begin{enumerate}[label={\upshape (\arabic*)}, leftmargin=2pc]}{\end{enumerate}}
\newenvironment{enumerate-(1)-r}{\begin{enumerate}[label={\upshape (\arabic*)}, leftmargin=2pc,resume]}{\end{enumerate}}

\newtheorem{theorem}{Theorem}[section]
\newtheorem{lemma}[theorem]{Lemma}
\newtheorem{corollary}[theorem]{Corollary}
\newtheorem{proposition}[theorem]{Proposition}

\newtheorem{problem}[theorem]{Problem}

\theoremstyle{definition}
\newtheorem{definition}[theorem]{Definition}

\newtheorem{example}[theorem]{Example}

\theoremstyle{remark}
\newtheorem{remark}[theorem]{Remark}

\begin{document}

\title[]{Borel structures on the space of left-orderings}

\date{}
\author[F.~Calderoni]{Filippo Calderoni}

\address{Department of Mathematics, Statistics, and Computer Science, University
of Illinois at Chicago, Chicago IL 60607, USA}
\email{fcaldero@uic.edu}

\author[A.~Clay]{Adam Clay}

\address{Department of Mathematics, 420 Machray Hall, University of
Manitoba, Winnipeg, MB, R3T 2N2}
\email{Adam.Clay@umanitoba.ca}

\thanks{Adam Clay was partially supported by NSERC grant RGPIN-2014-05465. The first author would like to thank Marcin Sabok for interesting discussions at the Arctic Set Theory Workshop~4.  We thank the anonymous referee for their valuable comments and suggestions.  Also we thank Uri Andrews, Turbo Ho, and Dino Rossegger for kindly pointing out an inaccuracy in Example~2.10 and for suggesting a way to fix it.}

\maketitle

\begin{abstract}
In this paper we study the Borel structure of the space of left-orderings $\LO(G)$ of a group $G$ modulo the natural conjugacy action, and by using tools from descriptive set theory we find many examples of countable left-orderable groups such that the quotient space \(\LO(G)/G\) is not standard. This answers a question of Deroin, Navas, and Rivas. We also prove that the countable Borel equivalence relation induced from the conjugacy action of \(\mathbb{F}_{2}\) on \(\LO(\mathbb{F}_{2})\) is universal, and leverage this result to provide many other examples of countable left-orderable groups $G$ such that the natural $G$-action on $\LO(G)$ induces a universal countable Borel equivalence relation.
\end{abstract}


\section{Introduction}

A group \(G\) is \emph{left-orderable} if it admits a total ordering \(<\) such that \(g < h\) implies \(fg < f h\) for all \(f, g, h \in G\), we call such total orderings \emph{left-orderings} of the group $G$.  We can equivalently define a group to be left-orderable if it admits a \emph{positive cone}, which is a subset $P$ of $G$ satisfying:
\begin{enumerate-(1)}
\item
\label{cond : lo1}
\(P \cdot P\subseteq P\);
\item
\label{cond : lo2}
\(P\sqcup P^{-1}\sqcup\{1\} = G\).
\end{enumerate-(1)}
There is a correspondence between left-orderings of $G$ and positive cones, by associating to each ordering $<$ of $G$ the set $P = \{ g \in G \mid g>1 \}$; and by associating to any subset \(P\subseteq G\) satisfying \ref{cond : lo1} and  \ref{cond : lo2} the left-invariant ordering of \(G\) defined by \(g<_{P}h\) if and only if \(g^{-1}h \in P\) for all \(g, h \in G\).  If a left-ordering $<$ of a group $G$ also happens to be right-invariant, that is, $g<h$ implies $gf<hf$ for all $f,g,h \in G$, then it is called a \emph{bi-ordering} of $G$.  Bi-orderings correspond precisely to the positive cones that additionally satisfy $gPg^{-1} \subset P$ for all $g \in G$.

Because of this association between left-orderings and positive cones, the set \(\LO(G)\) of all left-orderings of a group $G$ can be identified with a subspace of \(2^{G}\), which turns out to be closed, hence compact.  If one restricts attention to countable groups only, \(2^{G}\) becomes metrizable and \(\LO(G)\) is therefore a compact Polish space. 

In this paper, we consider the conjugacy action of a left-orderable group \(G\) on \(\LO(G)\), defined as follows:
Given an ordering of \(G\) with positive cone \(P\) and associated left-ordering $<$, for each element \(g \in G\) the image of \(<\) under \(g\) is the ordering \(<_{g}\) whose positive cone is \(gPg^{-1}\).  This is equivalent to
\[
f<_{g}h\quad \iff\quad gfg^{-1} < ghg^{-1}.
\]
Note that the bi-orderings of $G$ are precisely the fixed points of the $G$-action on $\LO(G)$, which turns out to be an action by homeomorphisms.  We will denote by \(E_\mathrm{lo}(G)\) the equivalence relation on \(\LO(G)\) whose classes are the orbits of the conjugacy action. Precisely, if \(P, Q \in \LO(G)\) we declare \((P,Q)\in E_\mathrm{lo}(G) \iff \exists g\in G \, (Q = gPg^{-1})\).

While the structure of this space and its natural $G$-action are not well-understood in general, both have been used to great effect.  For example, Witte-Morris used compactness of $\LO(G)$ to show that every left-orderable amenable group is locally indicable \cite{WM06}, a result which has since been generalized to locally invariant orderings and co-amenable subgroups~\cite{AntRiv}. (A group is \emph{locally indicable} if every finitely-generated nontrivial subgroup \(H\) of \(G\) admits a surjection \(H \to \mathbb{Z}\).) Similarly, Linnell~\cite{Lin11} used a clever combination of compactness of $\LO(G)$ and invariance of the derived sets under the $G$-action to show that every nontrivial left-orderable group has either $2^n$ left-orderings for some $n \geq 1$, or uncountably many left-orderings.  

The goal of this paper is to better understand $\LO(G)$ and its natural $G$-action by investigating the quotient Borel structure of  \(\LO(G)/G\). Deroin, Navas, and Rivas asked whether the quotient Borel structure of \(\LO(G)/G\) can be nonstandard~\cite[Question~2.2.11]{DerNavRiv}.  We first provide many examples of countable left-orderable groups \(G\) for which the answer is affirmative, by establishing a connection between the algebraic structure of $G$ and the Borel structure of  \(\LO(G)/G\).  Our approach uses tools from descriptive set theory to show:

\begin{theorem}
\label{thm : locally indicable}
If \(G\) is left-orderable but not locally indicable, then \(\LO(G)/G\) is not standard.
\end{theorem}

In a second approach to the problem, we use the theory of countable Borel equivalence relations to show that \(\LO(G)/G\) fails to be standard in the strongest possible sense for a large class of groups.   To better explain, we first recall some necessary background.

Suppose \(E\) and \(F\) are equivalence relations on the standard Borel spaces \(X\) and \(Y\), respectively. We say that \(E\) is \emph{Borel reducible} to \(F\) (written \(E\leq_{B} F\)) if there is a Borel map \(f\colon X\to Y\) such that
\[
x_{0}\mathbin{E}x_{1} \iff f(x_{0})\mathbin{F}f(x_{1}).
\]
Moreover when \(E\leq_{B} F\) and \(F\leq_{B} E\) we say that \(E\) and \(F\) are \emph{Borel equivalent} (written \(E\sim_{B}F\)).

We can take the statement ``\(E\leq_{B} F\)'' as a formal way of saying that the classification problem associated to \(E\), of determining whether to elements of \(X\) are \(E\)-equivalent, is not more complicated that the one associated to \(F\). In this precise sense Borel reducibility has been used to develop a complexity theory of definable equivalence relations. 
The main achievement in this area includes a series of anti-classification results showing that certain mathematical objects do not admit any reasonable classification.
For example, the work of Hjorth~\cite{Hjo99a} and Thomas~\cite{Tho03} shows that Baer's classification theorem cannot be extended to torsion-free abelian groups of rank \(k \geq 1\).
Moreover, Foreman and Weiss~\cite{ForWei} proved that conjugacy for measure-preserving transformations of the unit interval with Lebesgue measure is not Borel reducible to any isomorphism relation, hence is not classifiable by countable structures.

In particular, Borel reducibility is a tool that is fundamental to the analysis of the class of all \emph{countable Borel equivalence relations}.
Recall that Borel equivalence relation is said to be \emph{countable} if all of its equivalence classes are countable.
If a countable discrete group \(G\)  acts on a standard Borel space \(X\) in a Borel fashion, then the associated \emph{orbit equivalence relation}, whose classes are the orbits, is countable Borel. (In fact, by a theorem of Feldman and Moore~\cite{FelMoo} every countable Borel equivalence relation arises in that way.)

An important subclass of countable Borel equivalence relation consists of the ones that are Borel reducible to identity relation on \(\RR\). They are said to be \emph{smooth} (also \emph{tame}, or \emph{concretely classifiable}), and coincide with those whose quotient space carries a standard Borel structure. At the other extreme, there are countable Borel equivalence relations of maximal complexity with respect to $\leq_B$, which are called universal.
More precisely, a countable Borel equivalence
relation \(E\) is said to be \emph{universal} if  \(F \leq_{B} E\)
for every countable Borel
equivalence relation \(F\).

Between the extremes of smooth and universal, the class of countable Borel equivalence relations is rather complicated.  For instance, by a theorem of Adams and Kechris~\cite{AdaKec}
there are continuum many pairwise incomparable countable Borel equivalence relations up to Borel reducibility.

In the context of countable left-orderable groups, we can easily find examples of groups $G$ for which \(E_\mathrm{lo}(G)\) is smooth; torsion-free abelian groups are such an example.  Exploring the other extreme we show that there are also plenty of groups $G$ for which \(E_\mathrm{lo}(G)\) is universal, beginning with free groups.

\begin{theorem}
\label{maintheorem}
Let $\mathbb{F}_{n}$ denote the free group on $n$ generators.  If $n\geq2$ then \(E_\mathrm{lo}(\mathbb{F}_{n})\) is a  universal countable Borel equivalence relation.
\end{theorem}

Combining Theorem~\ref{maintheorem} with the results of Section~3, we deduce that \(E_\mathrm{lo}(G)\) is universal for a large class of countable groups, including hyperbolic surface groups, the pure braid groups \(P_{n}\) with \( n\geq 3\), right angled Artin groups, and many others.

\section{Generating non-smooth examples}

We generate our first examples of groups for which \(\LO(G)/G\) is nonstandard by appealing to the following equivalence. (E.g., see~\cite[Proposition~6.3]{KecMil}).

\begin{proposition}
For a countable Borel equivalence relation \(E\) the following are equivalent:
\begin{enumerate-(i)}
\item
\(E\) is smooth; i.e., there is a Borel map \(f\colon X \to \mathbb{R}\) such that
\[
x_{0}\mathbin{E}x_{1} \iff f(x_{0})=f(x_{1}).
\]
\item
The space \(X/E\) with the quotient Borel structure is standard.
\end{enumerate-(i)}
\end{proposition}

It is clear that \(E\) is smooth if and only if \(E\) is Borel reducible to the identity on \(\mathbb{R}\). Moreover, the class of smooth equivalence relations is downward closed with respect to \(\leq_{B}\). So,
whenever a nonsmooth equivalence relation \(E\) is Borel reducible to \(F\) defined on \(Y\), we obtain that \(F\) is nonsmooth, hence the quotient space \(Y/F\) is not standard.

The following proposition is a consequence of classical results in descriptive set theory (see~\cite[Corollary~3.5]{Hjo}).

\begin{proposition}
\label{Prop : dense orbit}
Let \(G\) be a countable group acting by homeomorphisms on a Polish space \(X\), and let \(E_{G}\) be the corresponding orbit equivalence relation. If there is a dense orbit and every orbit is meager, then \(E_{G}\) is not smooth.
\end{proposition}

With these results in hand, the fact that  \(\LO(G)/G\) is nonstandard for certain groups follows easily from existing results in the literature.

\begin{proposition}
If \(n\geq2\) or \(n=\infty\), then \(\LO(\mathbb{F}_{n})/\mathbb{F}_{n}\) is not standard.
\end{proposition}
\begin{proof}
A result of McLeary~\cite{McC85} ensures that \(\LO(\mathbb{F}_{n})\) is perfect.\footnote{The same result was also obtained using techniques from dynamics by Navas~\cite{Nav10} and, in a more general fashion, by Rivas~\cite{Riv12}.} Then each orbit of the conjugacy action  of \(\mathbb{F}_{n}\) in \(\LO(\mathbb{F}_{n})\) is meagre. Since \(\LO(\mathbb{F}_{n})\) admits a dense orbit (cf.~Clay \cite{Cla12} and Rivas~\cite{Riv12}),  it follows that \(\LO(\mathbb{F}_{n})/\mathbb{F}_{n}\) is not standard by Proposition~\ref{Prop : dense orbit}.
\end{proof}

A similar result follows from the recent work of Alonso, Brum and Rivas, showing that  \(\LO(G)\) admits a dense orbit under the natural $G$-action whenever $G$ is the fundamental group of a closed hyperbolic surface~\cite{ABR17}.

The examples we are able to generate using the existence of dense orbits are all examples of locally indicable groups that, in fact, turn out to yield universal countable Borel equivalence relations as well.  Our next technique produces groups $G$ which are not locally indicable, and for which  $E_\mathrm{lo}(G)$ is nonsmooth.

Recall that a \emph{minimal invariant set} $M$ for the action of a group $G$ on a space $X$ by homeomorphisms is a closed, $G$-invariant set $M \subset X$ satisfying the following:  If $C \subset X$ is any other closed, $G$-invariant set and $C \cap M \neq \emptyset$ then $ M \subset C$.  From this it follows that the orbit of every point in $M$ is, in fact, dense in $M$.

\begin{proposition}
\label{finite_orbit}
Suppose that $G$ is a countable group acting by homeomorphisms on a compact Polish space $X$ such that $E_G$ is smooth.  Then there exists a finite orbit.
\end{proposition}
\begin{proof}
Under the hypotheses of the proposition, 
let \[\mathcal{S} \coloneqq \{A \subseteq X \mid A\text{ is nonempty, closed and \(G\)-invariant}\}.\]
We note that \(\mathcal{S}\) is nonempty because \(X\in \mathcal{S}\) and it is  partially ordered  under inclusion.
A consequence of the compactness of \(X\) is that every chain \(\{A_{i}\}_{i\in I}\) satisfies the finite intersection property. It follows that \(\bigcap_{i\in I}A_{i}\) is nonempty and necessarily closed and \(G\)-invariant. Therefore, every chain in \(X\) admits has lower bound. Then, using Zorn's Lemma, there exists a minimal invariant set $M\subset X$ for the action of $G$. 

Since the action of $G$ on $X$ induces a smooth countable Borel equivalence relation, so does the action of $G$ on $M$.  Note that every orbit of every point in $M$ is dense, so by Proposition~\ref{Prop : dense orbit}, $M$ must have a nonmeagre orbit.  

Suppose that $x_0 \in X$ is the point whose orbit is nonmeagre.  Writing the orbit of $x_0$ as the union of singletons $\bigcup_{g \in G} \{ g \cdot x_0 \}$ expresses that the orbit as a countable union of nowhere dense sets, unless one of the singletons $\{ g \cdot x_0 \} $ is open in $M$. We conclude that $M$ admits an isolated point, and since every orbit is dense in $M$, this means $M$ must in fact consist of a single orbit.

But now $M$ is a countable, closed (hence compact) subset and every point in $M$ is  isolated (since all points in $M$ are contained in a single orbit, which contains an isolated point).  This is not possible if $M$ is infinite.  
\end{proof}

Recall that a left-ordering $<$ of a group $G$ is \emph{Conradian} if and only if for every pair of positive elements $g, h \in G$, there exists $n \in \mathbb{N}$ such that $g < hg^n$. It is a theorem of Brodskii~ \cite{Brodskii84} that a group \(G\) is Conradian left-orderable if and only if \(G\) is locally indicable.

\begin{proof}[Proof of Theorem~\ref{thm : locally indicable}]
Suppose that the orbit of $P \in \mathrm{LO}(G)$ is finite.  Then, given any $g \in G$, there exists $n \in \mathbb{N}$ such that $g^{-n} P g^n = P$.  In particular, for any pair of elements $g, h \in P$ this implies that $g^{-n} h g^n \in P$, so that the left-ordering associated to $P$ satisfies $g^n < hg^n$.  Since $g$ is positive, this implies $g <_P \dots <_Pg^n <_P hg^n$, so that $P$ determines a Conradian ordering.  The result then follows from Proposition~\ref{finite_orbit}.
\end{proof}

\begin{remark}
Note that the arguments in the proof of Theorem~\ref{thm : locally indicable} in fact imply that any group $G$ for which $E_\mathrm{lo}(G)$ is smooth must be virtually bi-orderable.
\end{remark}

There are many groups which satisfy the hypotheses of Proposition~\ref{thm : locally indicable}, 
such as the braid groups \(B_{n}\) for \(n \geq 5\) (by \cite{GL69}, their commutator subgroups are finitely generated and perfect), or the fundamental group of many compact $3$-manifolds (e.g. see~\cite{BRW05} for plenty of examples).

On the other hand, as mentioned in the introduction, if $G$ is torsion-free abelian then $E_\mathrm{lo}(G)$ is smooth since the action of $G$ on $\LO(G)$ is trivial.  For similar trivial reasons, if $G$ is a so-called Tararin group (meaning that $\LO(G)$ is finite) then $G$ is nonabelian, yet $E_\mathrm{lo}(G)$ is smooth.  The next example shows that smoothness of $E_\mathrm{lo}(G)$ is more subtle than either $G$ being abelian or $\LO(G)$ being finite, as it exhibits a nonabelian group $G$ for which $\LO(G)$ is infinite, the action of $G$ on $\LO(G)$ is nontrivial, and $E_\mathrm{lo}(G)$ is smooth. We first recall the following standard definition.

\begin{definition}
Let \(G\) be a group equipped with a left-ordering $<$. A subgroup $C$ of $G$ is \emph{convex relative to $<$} if whenever $g, h\in C$ and $f \in G$ with $g<f<h$, then $f \in C$.  A subgroup  \(C\subseteq G\) is \emph{left relatively
convex} in \(G\) if \(C\) is convex relative to some left ordering of \(G\).
\end{definition}

Left-relatively convex subgroups were recharacterized by Antol\'{i}n and Rivas~\cite[Lemma~2.1]{AntRiv} as follows.

\begin{proposition}
\label{prop : AntRiv}
A subgroup 
\(C\) of $G$ is left relatively convex if and only if there exists \(Q\subseteq G\) satisfying:
\begin{enumerate}
\item \label{AntRiv : i}
\( QQ \subseteq Q \);
\item \label{AntRiv : ii}
\(CQC\subseteq Q\);
\item \label{AntRiv : iii}
\(Q\cup Q^{-1}\cup C = G\).
\end{enumerate}
In this case, if \(P\in \LO(C)\), then \(P\cup Q\) is a positive cone in \(G\).
Conversely if \(R \in \LO(G)\) and $R$ decomposes as \(R = {P\cup Q}\) with \(P\in \LO(C)\), then \(Q\) satisfies~\ref{AntRiv : i}--\ref{AntRiv : iii}.
\end{proposition}

This turns out to be much more useful for our purposes, and so will be used without reference in the examples below, as well as in the proofs of Section 3.

\begin{example}
Let $\langle z \rangle$ be an infinite cyclic group whose generator $z$ acts on the abelian group $\ZZ \times \ZZ$ by the matrix $\begin{psmallmatrix}-1 & 0\\0 & -1\end{psmallmatrix}$.  Let $G$ denote the semidirect product $(\ZZ \times \ZZ)\rtimes \langle z \rangle$.
Then $z$ satisfies $z^{-1}xz = x^{-1}$ for all $x \in \ZZ \times \ZZ$.  We first note that this implies $\ZZ \times \ZZ$ is convex in every left-ordering of $G$.

To see this, suppose that $1<z <x^k$ for some $k \in \ZZ$, for some left-ordering of $G$.  Then $1<z^{-1}x^k$, and hence $1<z^{-1}x^kz$ as the right hand side is a product of positive elements.  But $z^{-1}x^kz = x^{-k}$ is negative, a contradiction.  Thus if $z>1$ then $x^k <z$ for all $k \in \mathbb{Z}$.  By similar arguments we conclude $z^{-1} < x^k <z$ for all $k \in \ZZ$ whenever $z >1$, and $z < x^k <z^{-1}$ for all $k \in \ZZ$ whenever $z <1$.  It follows that $\ZZ \times \ZZ$ is convex in every left-ordering of $G$.

Thus every left-ordering of $G$ arises lexicographically from the short exact sequence 
\[ 1 \rightarrow \ZZ \times \ZZ \rightarrow G \rightarrow \langle z \rangle \rightarrow 1. 
\]
That is, if $P$ is the positive cone of a left-ordering of $G$ then $P = R \cup Q$, where $R$ is the positive cone of an ordering of $\ZZ \times \ZZ$ and $Q = P \setminus R$.  The image of $P$ under the action of $x z^k \in G$ is again $P$ if $k$ is even, for then $x z^k$ is central; and it is $R^{-1} \cup Q$ if $k$ is odd.
As such, each orbit of the $G$-action on $\LO(G)$ consists of exactly two elements.  Consequently $E_\mathrm{lo}(G)$ is smooth \cite[Example 6.1]{KecMil}.
\end{example}

We close out this section by presenting a construction of a group $T_{\infty}$ such that $E_\mathrm{lo}(T_{\infty})$ is as simple as possible, yet not smooth.  Recall that $E_0$ is the equivalence relation of ``eventual equality'' on the set of sequences $\{0,1\}^{\mathbb{N}}$.

\begin{theorem}[(Glimm-Effros dichotomy~\cite{HarKecLou})]
If \(E\) is a countable Borel equivalence relation, then either:
\begin{enumerate-(i)}
\item
\(E\) is smooth; or
\item
\(E_{0} \leq_{B} E\). (In fact, \(E_{0}\sqsubseteq_{B} E\).)
\end{enumerate-(i)}
\end{theorem}

In the precise sense above, it is $E_0$ which is ``as simple as possible, yet not smooth''. 

\begin{example}
Here is an example of a group $T_{\infty}$, such that $E_\mathrm{lo}(T_{\infty})$ is Borel equivalent to $E_0$.  For each $n \geq 1$, let $T_n$ denote the group 
\[ \langle x_1, x_2, \ldots, x_n \mid x_i x_{i-1}x_i^{-1} = x_{i-1}^{-1} \mbox{ for $1<i \leq n$ and } x_i x_j = x_j x_i \mbox{ for $|i -j|>1$} \rangle.
\]
Then each $T_n$ is a Tararin group, i.e. it is a group admitting $2^n$ left-orderings (see \cite[Theorem 5.2.1]{KM96}).  The convex subgroups of each left-ordering of $T_n$ are precisely the subgroups $T_i$, $i \leq n$, together with the trivial subgroup $\{ id \}$.  Thus every ordering of $T_n$ is determined by the choice of signs for the generators.  

Now consider the group $T_{\infty}$ given by the following presentation
\[ \langle x_1, x_2, \ldots \mid x_i x_{i-1}x_i^{-1} = x_{i-1}^{-1} \mbox{ for $1<i \leq k$ and } x_i x_j = x_j x_i \mbox{ for $|i -j|>1$} \rangle.
\]
For every left-ordering of $T_{\infty}$, the convex subgroups of $T_{\infty}$ are precisely the subgroups $T_i$.  As such, the orderings of $T_{\infty}$ are in bijective correspondence with sequences $ (\varepsilon_i) \in \{0, 1\}^{\mathbb{N}}$ that encode the signs of the generators according to the rule: $x_i > 1$ if and only if $\varepsilon_i = 1$.  Moreover, it is not hard to see that the conjugation action of $T_{\infty}$ on the set $\mathrm{LO}(T_{\infty})$ yields an action of $T_{\infty}$ on $\{0,1\}^{\NN}$ given by: $x_j \cdot (\varepsilon_i)$ is the same as $(\varepsilon_i)$ in every entry except the $(j-1)^{th}$ position, which has been changed.

Consequently two left-orderings of $T_{\infty}$ are in the same orbit if and only if their corresponding sequences in $\{0, 1\}^{\mathbb{N}}$ are in the same orbit of $T_{\infty}$ under the action described above. This happens if and only if the sequences are eventually equal.
\end{example}

\section{Some reducibility results}
\label{reducibility_results}

In this section, we prepare a variety of results that are necessary for producing examples of groups $G$ for which $E_\mathrm{lo}(G)$ is universal, and in particular show that $E_\mathrm{lo}(\mathbb{F}_n) \sim_B E_\mathrm{lo}(\mathbb{F}_\infty)$ for all $n >1$.  

\begin{proposition}
\label{Thm : relconvex}
If \(C\) is left relatively convex in \(G\) and 
\begin{equation}
\label{weakly malnormal}
\tag{\(\ast\)}
\text{for all } h\in G\quad hCh^{-1}\subseteq C \implies h\in C,
\end{equation}
then \(E_\mathrm{lo}(C) \leq_{B} E_\mathrm{lo}(G)\).
\end{proposition}

\begin{proof}
Define \(f\colon \LO(C) \to \LO(G)\) by setting \(f(P)= P \cup Q\). Clearly \(f\) is Borel, in fact it is continuous.
If \(g\in C\), and \(P,R \in \LO(C)\) such that \(gPg^{-1} = R\), then
\(gf(P)g^{-1} = f(R)\) because
\[
g f(P) g^{-1}    =  g (P\cup Q) g^{-1} 
		    	=  g P g^{-1} \cup gQg^{-1} 
		    	=  R \cup gQg^{-1} = R \cup Q.
			\]
The last equality holds by part~\ref{AntRiv : ii} of Proposition~\ref{prop : AntRiv} as \(Q\subseteq g^{-1}Qg\) implies \(gQg^{-1} \subseteq Q\).

On the other hand, we next claim that if \(h\in G\) and \(h(R\cup Q)h^{-1} = P\cup Q\) for some \(R,P\in \LO(C)\), then \(h\in C\)
and \(hRh^{-1} = P\).

Since \(h(R\cup Q)h^{-1}= hRh^{-1} \cup hQ h^{-1}\) is a positive cone, and
\(hRh^{-1}\subseteq hCh^{-1}\) is a positive cone, then \(hCh^{-1}\) is convex relative to \(<_{P\cup Q}\).
Moreover \(P\cup Q \in \LO(G)\) and \(P\in \LO(C)\) then \(C\) is convex relative to \(<_{P\cup Q}\).
It follows that  either \(hCh^{-1}\subseteq C\) or \(C \subseteq hCh^{-1}\), thus \(h\in C\) by \eqref{weakly malnormal}.
Since \(hQh^{-1} = Q\), we have
\(hRh^{-1}\cup Q = P\cup Q\), which implies \(hRh^{-1} = P\) as desired.
\end{proof}

\begin{proposition} 
\label{prop:free_grps}
Let $n \geq 2$.  Then
\(E_\mathrm{lo}(\mathbb{F}_{2}) \leq_{B} E_\mathrm{lo}(\mathbb{F}_{n})\)
and 
\(E_\mathrm{lo}(\mathbb{F}_{2}) \leq_{B} E_\mathrm{lo}(\mathbb{F_{\infty}})\).
\end{proposition}

\begin{proof}
By \cite[Corollary 20]{ADZ15}, there exists a left-ordering of $\mathbb{F}_{n}$ such that $\mathbb{F}_{2}\subseteq \mathbb{F}_{3}\subseteq \cdots \subseteq \mathbb{F}_{n}$ are convex.
Since \(\mathbb{F}_{2} \) is malnormal in \(\mathbb{F}_{n}\), then \(\mathbb{F}_{2}\) satisfies \eqref{weakly malnormal}. Thus the result follows from Proposition~\ref{Thm : relconvex}.  The same proof holds when $n = \infty$.
\end{proof}


\begin{proposition}
\label{prop:infty to 2}
\(E_\mathrm{lo}(\mathbb{F}_{\infty}) \leq_{B} E_\mathrm{lo}(\mathbb{F}_{2})\).
\end{proposition}
\begin{proof}
Suppose that $\mathbb{F}_2$ has generators $a, b$ and that \(h\colon \mathbb{F}_{2} \to \mathbb{Z}\) is the homomorphism defined by
\(h(a)=1\) and \(h(b)=0\).   Set \(H = \ker h\), and note that \(\{ a^{n} \mid n\in \mathbb{Z}\}\) is a set of coset representatives of \(H\) in \(\mathbb{F}_{2}\), from which Reidemeister-Schreier yields \(S=\{a^{n}ba^{-n}\mid n\in \mathbb{Z}\}\) as a basis of \(H\). Set $x_n =a^{n}ba^{-n}$ for ease of exposition. 

Noting that $ax_na^{-1} = x_{n+1}$, we can argue exactly as in \cite[Theorem~3.3]{Tho15} to find an infinite subset \(C\subseteq S\) satisfying \(|a^{\ell}Ca^{-\ell}\cap C|\leq 1\) for all \(0\neq \ell\in \mathbb{Z}\).  The subgroup \(\langle C\rangle\) of $\mathbb{F}_2$ is evidently isomorphic to \(\mathbb{F}_{\infty}\), and we claim that it satisfies the hypotheses of Proposition~\ref{Thm : relconvex}. Verifying this claim will complete the proof.

First, \(H\) is left relatively convex in \(\mathbb{F}_{2}\) because its quotient is left-orderable. Next, \(\langle C\rangle\) is left relatively convex in \(H\) by \cite[Corollary 20]{ADZ15} because we can express \(H\) as a free product with \(\langle C\rangle\) as one of the factors. Since left relative convexity is transitive, it follows that \(\langle C\rangle\) is left relatively convex in \(\mathbb{F}_{2}\).

Now, given any \(w\in \mathbb{F}_{2}\setminus \langle C\rangle\), write \(w=ga^{\ell}\) for some \(g\in H\) and \(\ell \in \mathbb{Z}\).
If \(\ell \neq 0\), then there exists \(x_{i}\in C\)
such that \(a^{\ell}x_{i}a^{-\ell}\notin \langle C\rangle\). It follows that
 \(wx_iw^{-1}=ga^{\ell}x_{i}a^{-\ell}g^{-1} \notin \langle C\rangle\). 
 Otherwise, suppose \(\ell=0\) and $w\langle C\rangle w^{-1} \cap \langle C\rangle \neq \{ 1\}$. Then $w \in C$ follows by observing that $w = g \in H$, and as $\langle C\rangle$ is one of the factors in an expression of $H$ as a free product, it is malnormal in $H$.
\end{proof}


\begin{proposition}
\label{reduction_prop}
Suppose that there is a short exact sequence of groups $$1 \rightarrow K \stackrel{i}{\rightarrow} G \stackrel{q}{\rightarrow} H \rightarrow 1,$$ that $K$ and $H$ are left-orderable and that $K$ admits a positive cone $P$ such that $gi(P)g^{-1} = P$ for all $g \in G$.  Then
\[ E_\mathrm{lo}(H) \leq_{B} E_\mathrm{lo}(G).
\]
\end{proposition}
\begin{proof}
Fix a positive cone $P \in \LO(K)$ as in the statement of the theorem.  Define $f \colon \LO(H) \rightarrow \LO(G)$ by $f(Q) = q^{-1}(Q) \cup P$.
Given $h \in H$, suppose that $R, Q \in \LO(H)$ and satisfy $hRh^{-1} = Q$.  Choose $g \in G$ such that $q(g) = h$.  Then 
\[ gf(R)g^{-1} = g(q^{-1}(R) \cup P)g^{-1} = g(q^{-1}(R))g^{-1}  \cup gPg^{-1} = q^{-1}(Q) \cup P = f(Q).
\]

On the other hand, suppose that $gf(R)g^{-1} = f(Q)$ for positive cones $R, Q \in \LO(H)$.  Then $g(q^{-1}(R) \cup P)g^{-1} = g(q^{-1}(R))g^{-1} \cup P = q^{-1}(Q) \cup P$.  Applying the homomorphism $q$ gives $q(g) Rq(g)^{-1} = Q$.
\end{proof}

\begin{corollary}
\label{corquotient}
Suppose that there is a short exact sequence of groups $$1 \rightarrow K \stackrel{i}{\rightarrow} G \stackrel{q}{\rightarrow} H \rightarrow 1,$$ where $H$ is left-orderable and $G$ is bi-orderable.  Then
\[
E_\mathrm{lo}(H) \leq_{B} E_\mathrm{lo}(G).
\]
\end{corollary}
\begin{proof}
A choice of positive cone $P \in \LO(K)$ as in Proposition \ref{reduction_prop} is always possible when $G$ is bi-orderable.
\end{proof}

\begin{corollary}
\label{corollary : ses countable G} Suppose that $G$ is a countable, left-orderable group.  Then 
\(E_\mathrm{lo}(G) \leq_{B} E_\mathrm{lo}(\mathbb{F}_{\infty})\).
\end{corollary}

\begin{proof}
Let \(K =\ker p \triangleleft \mathbb{F}_{\infty}\) where
\(p\colon \mathbb{F}_{\infty}\to G\) is a surjective homomorphism, this exists since $G$ is countable. Then consider the short exact sequence $ 1 \rightarrow K \rightarrow \mathbb{F}_\infty \rightarrow G \rightarrow 1$. Since \(\mathbb{F}_{\infty}\) is bi-orderable, it follows from Corollary~\ref{corquotient} that \(E_\mathrm{lo}(G)\) is Borel reducible to \(E_\mathrm{lo}(\mathbb{F}_{\infty})\). 
\end{proof}

\section{Universality of \(E_\mathrm{lo}(\mathbb{F}_{2})\)}

Let  \(E(\mathbb{F}_{2},2)\) be the equivalence relation arising from the left-shift action of \(\mathbb{F}_{2}\) on \(2^{\mathbb{F}_{2}} = \mathcal{P}(\mathbb{F}_{2})\). We shall use the fact that \(E(\mathbb{F}_{2},2)\) is a universal countable Borel equivalence relation (see~\cite[Proposition~1.8]{DouJacKec}). We begin with a few preliminaries concerning lexicographic bi-orderings of direct sums and wreath products, which will be used in the proof.

\begin{lemma}
\label{lex ords}
Suppose that $\{G_i\}_{i \in I}$ is a family of bi-orderable groups, and that for each $i \in I$ the subset $P_i \subset G_i$ is a positive cone of a bi-ordering of $G$.  Suppose further that $\prec$ is a total ordering of $I$, and for each $f \in \bigoplus_{i\in I} G_i$, set $i_f = \min_{\prec} \{i \mid f(i) \neq 1 \}$.  Then $$P = \{ f \in \bigoplus_{i\in I} G_i \mid f(i_f) \in P_{i_f} \}$$ is the positive cone of a bi-ordering of $\bigoplus_{i\in I} G_i$.
\end{lemma}
\begin{proof} 
Set $G = \bigoplus_{i\in I} G_i$.  It is clear from the definition that $P \cup P^{-1} = G \setminus \{ 1\}$, and that $P \cap P^{-1} = \emptyset$.  

To see that $P$ is a semigroup, suppose that $g, f \in P$.  If $i_f \neq i_g$, note that $i_{fg} = \min_{\prec}\{i_f, i_g\}$ and consider two cases.  In the case that $i_{fg} = i_f$, then $fg(i_{fg}) = f(i_f)  g(i_f) = f(i_f) \in P_{i_f} = P_{i_{fg}}$, so that $fg \in P$.   The case when $i_{fg} = i_g$ is similar.  When $i_f = i_g = i_{fg}$, then $fg(i_{fg}) =f(i_f)  g(i_g)$ is a product of elements in the positive cone $P_{i_{fg}}$, and so again lies in $P_{i_{fg}}$.

Last we check conjugation invariance of $P$.  Suppose that $f \in P$ and $g \in G$ is arbitrary, and note that $i_{gfg^{-1}} = i_f$. So upon observing that $gfg^{-1}(i_{gfg^{-1}}) = g(i_f) f(i_f) g(i_f)^{-1} \in P_{i_f}$ (since $g(i_f)P_{i_f}g(i_f)^{-1} \subset P_{i_f})$, we can conclude that $gPg^{-1} \subset P$ for all $g \in G$.
\end{proof}

Next, let \(C\) be a countable bi-orderable group. We will consider the restricted wreath product \(C\wr \mathbb{F}_{2}\) of \(\mathbb{F}_{2}\) and \(C\), recall this is defined as follows.  For each function \(f\colon \mathbb{F}_{2}\to C\) set
\[
S(f) = \setm{x\in \mathbb{F}_{2}}{f(x)\neq 1}.
\]
Then consider the group
\[
B=\setm{f\colon \mathbb{F}_{2}\to C}{S(f)\text{ is finite}}
\]
with point-wise multiplication, i.e. \(fg(x) = f(x)g(x)\) for all \({f,g}\in B\) and \(x\in \mathbb{F}_{2}\).
Clearly \(\mathbb{F}_{2}\) acts on \(B\) by
\[
a\cdot f(x) = f(a^{-1}x).
\]
The restricted wreath product \(C\wr \mathbb{F}_{2}\) is defined as \(B\rtimes \mathbb{F}_{2}\), and we will denote it by $W$.  For each \(x \in \mathbb{F}_{2}\), let
\[
C_{x} = \setm{f\in B}{f(y)= 1\text{ for all \(y\neq x\in \mathbb{F}_{2}\)}}.
\]
From this, it is clear that \(B= \bigoplus_{x\in\mathbb{F}_{2}} C_{x}\); and \(aC_{x}a^{-1}=C_{ax}\) for all \(a,x\in\mathbb{F}_{2}\).
We use this expression for $B$ in the proof below, as well as the short exact sequence
\[
1 \to B\stackrel{i}\to W \stackrel{q}\to \mathbb{F}_{2} \to 1. \tag{$*$}
\]

Then Theorem~\ref{maintheorem} is a consequence of the following:

\begin{proposition}
\label{Prop : reductionwreath}  With notation as above, 
\(E(\mathbb{F}_{2},2) \leq_{B}  E_\mathrm{lo}(W)\). Thus, \(E_\mathrm{lo}(W)\) is a universal countable Borel equivalence relation.
\end{proposition}

\begin{proof}
Fix a bi-invariant positive cone \(P\) of \(C\), and let \(P^{-1}\) be the complement of \(P\) in
\(C\setminus\{1\}\).  Clearly \(P^{-1}\) is also the positive cone of a bi-ordering on \(C\).  Next, fix a left-ordering \(\prec\) on \(\mathbb{F}_{2}\).

Then, given \(A\subseteq \mathbb{F}_{2}\), for each \(x\in \mathbb{F}_{2}\) set

\[
P_{x} =
\begin{cases}
P &\text{if \(x\in A\),}\\
P^{-1}&\text{otherwise}
\end{cases}
\]
and define \(R_{A}\) as the positive cone of the lexicographic order on \(B= \bigoplus_{x\in\mathbb{F}_{2}} C_{x}\) with respect to \(\prec\) and the family of positive cones \(\{P_{x}\}_{x \in \mathbb{F}_2}\), as in Lemma~\ref{lex ords}.

Finally, let \(Q\) be the positive cone of a bi-ordering on \(\mathbb{F}_{2}\), and define a map $\varphi\colon \mathcal{P}(\mathbb{F}_{2}) \rightarrow \mathrm{LO}(W)$ by \(\varphi(A) = R_{A} \cup q^{-1}{(Q)}\).  Note that $\varphi(A)$ is the positive cone of a lexicographic left-ordering of $W$, constructed from the short exact sequence~$(*)$.
\big(Namely, for all \(w,w'\in W\) we have \(w<_{\varphi(A)} w'\) if and only if
\(q(w)<_Q q(w')\) or else \(q(w) = q(w')\), thus \(w^{-1}w \in \ker q = B\), 
and \(w^{-1}w \in R_A\).\big)

We claim that if $h \in \mathbb{F}_2$, then $h R_A h^{-1} = R_{hA}$. In fact, if \(b\in R_{A}\), then \(b\in \bigoplus_{x\in F} C_{x}\) for some finite \(F\subseteq \mathbb{F}_2\)
and \(b(x_0) \in P_{x_0}\) where \(x_0 = \min_{\prec} F\).
It follows that \(hbh^{-1} \in \bigoplus_{x\in hF} C_{x}\), \(hx_0 = \min_{\prec} hF\) and \(hbh^{-1}(hx_0) = b(x_0) \in P_{hx_0}\), so that $hbh^{-1} \in R_{hA}$.  The inclusion $h R_A h^{-1} \subset R_{hA}$ of positive cones implies $h R_A h^{-1} = R_{hA}$.

Now we check that $\varphi$ is a Borel reduction from \(E(\mathbb{F}_{2},2)\)  to \(E_\mathrm{lo}(W)\). Let \(g = hb \in W \) be any element with \(h\in \mathbb{F}_{2}\) and \(b\in B\).
We compute
\begin{align*}
hb\varphi(A) b^{-1} h^{-1} &=  hbR_{A} b^{-1} h^{-1} \cup hbq^{-1}(Q)b^{-1}h^{-1}\\
&=  hR_{A}  h^{-1} \cup q^{-1}(Q)\\
&= R_{hA} \cup q^{-1}(Q)\\
&= \varphi(hA),
\end{align*}
where the second equality follows from the fact that $R_A$ and $Q$ are positive cones of bi-orderings of $B$ and $\mathbb{F}_2$ respectively.
\end{proof}

From the previous theorem we are now able to prove Theorem~\ref{maintheorem}.

\begin{proof}[Proof of Theorem~\ref{maintheorem}]
By Proposition~\ref{prop:infty to 2} it suffices to observe that \(E_\mathrm{lo}(W)\) is Borel reducible to \(E_\mathrm{lo}(\mathbb{F}_{\infty})\),
which follows from Corollary~\ref{corollary : ses countable G} since \(W\) is countable.
\end{proof}

With the results of Section \ref{reducibility_results}, and universality of \(E_\mathrm{lo}(\mathbb{F}_{n})\) for all $n \geq 2$, it is relatively straightforward to produce left-orderable groups $G$ for which $E_\mathrm{lo}(G)$ is universal.

\begin{corollary} The relation 
 \(E_\mathrm{lo}(G)\) is a universal countable Borel equivalence relation whenever $G$ is:
\begin{enumerate-(i)}
\item 
\label{universal : 0} a left-orderable direct product of groups $A \times B$ such that $E_\mathrm{lo}(A)$ is universal,  
\item
\label{universal : i} a left-orderable free product of groups $A*B$ such that $E_\mathrm{lo}(A)$ is universal,
\item
\label{universal : ii}
a hyperbolic surface group,
\item
\label{universal : iii}
a pure braid group on \(n \geq 3\) strands,
\item 
\label{universal : iiii}
a nonabelian right-angled Artin group.
\end{enumerate-(i)}
\end{corollary}
\begin{proof}
With \(G\) as in \ref{universal : 0} the result follows from Proposition \ref{reduction_prop}.
If \(G\) is as in \ref{universal : i} the result follows from Proposition~\ref{Thm : relconvex}, and the fact that every factor in a free product is both relatively convex and malnormal. Instead, if \(G\) is as in  \ref{universal : ii}--\ref{universal : iiii}, then \(G\) is bi-orderable: bi-orderability of surface groups appears in \cite{RW01}, pure braid groups in~\cite[Chapter XV]{DDRW08}, while right-angled Artin groups are residually torsion-free nilpotent, hence bi-orderable. We can then use Corollary \ref{corquotient} as hyperbolic surface groups, pure braid groups $P_n$ ($n \geq 3$), and nonabelian right angled Artin groups admit free nonabelian quotients. (See \cite[Corollary 3.7]{CFR12} for a proof that $P_n$, $n \geq3$ has a nonabelian free quotient.)
\end{proof}

\begin{remark}
The fact that  \ref{universal : ii}--\ref{universal : iiii} are not standard does not follow from Theorem~\ref{thm : locally indicable}, since these groups are bi-orderable hence locally indicable.
\end{remark}

Having found examples of groups $G$ for which $E_\mathrm{lo}(G)$ is smooth, for which $E_\mathrm{lo}(G) \sim_B E_0$ and for which $E_\mathrm{lo}(G)$ is universal, the following is the natural next step.

\begin{problem}
Find a group $G$ such that $E_\mathrm{lo}(G)$ is intermediate.
\end{problem}



\begin{thebibliography}{99}
%
%
\bibitem{AdaKec}
Scott Adams and Alexander~S. Kechris.
\newblock Linear algebraic groups and countable {B}orel equivalence relations.
\newblock {\em J. Amer. Math. Soc.}, 13(4):909--943 (electronic), 2000.

\bibitem{ABR17}
Juan Alonso, Joaqu\'{\i}n Brum, and Crist\'{o}bal Rivas.
\newblock Orderings and flexibility of some subgroups of
  {$\mathrm{Homeo}_+(\Bbb R)$}.
\newblock {\em J. Lond. Math. Soc. (2)}, 95(3):919--941, 2017.

\bibitem{ADZ15}
Yago Antol\'{\i}n, Warren Dicks, and Zoran \v{S}uni\'{c}.
\newblock Left relatively convex subgroups.
\newblock In {\em Topological methods in group theory}, volume 451 of {\em
  London Math. Soc. Lecture Note Ser.}, pages 1--18. Cambridge Univ. Press,
  Cambridge, 2018.

\bibitem{AntRiv}
Yago Antol\'{\i}n and Crist\'obal Rivas.
\newblock The space of relative orders and a generalization of Morris
  indicability theorem.
\newblock {\em Journal of Topology and Analysis},  13 (01):75--85,  2021.


\bibitem{BRW05}
Steven Boyer, Dale Rolfsen, and Bert Wiest.
\newblock Orderable 3-manifold groups.
\newblock {\em Ann. Inst. Fourier (Grenoble)}, 55(1):243--288, 2005.

\bibitem{Brodskii84}
Sergei~D. Brodski\u{\i}.
\newblock Equations over groups, and groups with one defining relation.
\newblock {\em Sibirsk. Mat. Zh.}, 25(2):84--103, 1984.

\bibitem{Cla12}
Adam Clay.
\newblock Free lattice-ordered groups and the space of left orderings.
\newblock {\em Monatsh. Math.}, 167(3-4):417--430, 2012.

\bibitem{CFR12}
Daniel~C. Cohen, Michael Falk, and Richard Randell.
\newblock Pure braid groups are not residually free.
\newblock In A.~Bjorner, F.~Cohen, C.~De~Concini, C.~Procesi, and M.~Salvetti,
  editors, {\em Configuration Spaces}, pages 213--230, Pisa, 2012. Scuola
  Normale Superiore.

\bibitem{DDRW08}
Patrick Dehornoy, Ivan Dynnikov, Dale Rolfsen, and Bert Wiest.
\newblock {\em Ordering braids}, volume 148 of {\em Mathematical Surveys and
  Monographs}.
\newblock American Mathematical Society, Providence, RI, 2008.

\bibitem{DerNavRiv}
Bertrand Deroin, Andr\'{e}s Navas, and Crist\'{o}bal Rivas.
\newblock {\em Groups, orders, and dynamics}.
\newblock Preprint, available at \href{https://arxiv.org/abs/1408.5805}{arXiv:1408.5805}, 2016.

\bibitem{DouJacKec}
Randall Dougherty, Stephen Jackson, and Alexander~S. Kechris.
\newblock The structure of hyperfinite borel equivalence relations.
\newblock {\em Transactions of the American Mathematical Society},
  341(1):193--225, 1994.

\bibitem{FelMoo}
Jacob Feldman and Calvin~C. Moore.
\newblock Ergodic equivalence relations, cohomology, and von Neumann algebras.
\newblock {\em Bull. Amer. Math. Soc.}, 81(5):921--924, 09 1975.

\bibitem{ForWei}
Matthew Foreman, and Benjamin Weiss.
\newblock An anti-classification theorem for ergodic
measure preserving transformations.
\newblock {\em J. Eur. Math. Soc.} 6, 277--292, 2004.

\bibitem{GL69}
E.~A. Gorin and V.~Ja. Lin.
\newblock Algebraic equations with continuous coefficients, and certain
  questions of the algebraic theory of braids.
\newblock {\em Mat. Sb. (N.S.)}, 78 (120):579--610, 1969.

\bibitem{HarKecLou}
L.~A. Harrington, A.~S. Kechris, and A.~Louveau.
\newblock A {G}limm-{E}ffros dichotomy for {B}orel equivalence relations.
\newblock {\em J. Amer. Math. Soc.}, 3(4):903--928, 1990.

\bibitem{Hjo99a}
Greg Hjorth.
\newblock Around nonclassifiability for countable torsion free abelian groups.
\newblock 75:269--292, 1999.

\bibitem{Hjo}
Greg Hjorth.
\newblock {\em Classification and Orbit Equivalence Relations}, volume~75 of
  {\em Mathematical Surveys and Monographs}.
\newblock American Mathematical Society, Providence, RI, 2000.

\bibitem{KecMil}
Alexander~S. Kechris and Benjamin Miller.
\newblock {\em Topics in orbit equivalence relations}.
\newblock Lecture notes in Mathematics. Springer-Verlag Berlin Heidelberg,
  2004.

\bibitem{KM96}
Valeri\u{\i}~M. Kopytov and Nikola\u{\i}~Ya. Medvedev.
\newblock {\em Right-ordered groups}.
\newblock Siberian School of Algebra and Logic. Consultants Bureau, New York,
  1996.

\bibitem{Lin11}
Peter~A. Linnell.
\newblock The space of left orders of a group is either finite or uncountable.
\newblock {\em Bull. Lond. Math. Soc.}, 43(1):200--202, 2011.

\bibitem{McC85}
Stephen~H. McCleary.
\newblock Free lattice-ordered groups represented as o-2 transitive
  l-permutation groups.
\newblock {\em Transactions of the American Mathematical Society},
  290(1):69--79, 1985.

\bibitem{WM06}
Dave~Witte Morris.
\newblock Amenable groups that act on the line.
\newblock {\em Algebr. Geom. Topol.}, 6:2509--2518, 2006.

\bibitem{Nav10}
Andr\'{e}s Navas.
\newblock On the dynamics of (left) orderable groups.
\newblock {\em Ann. Inst. Fourier (Grenoble)}, 60(5):1685--1740, 2010.

\bibitem{Riv12}
Crist\'{o}bal Rivas.
\newblock Left-orderings on free products of groups.
\newblock {\em J. Algebra}, 350:318--329, 2012.

\bibitem{RW01}
Dale Rolfsen and Bert Wiest.
\newblock Free group automorphisms, invariant orderings and topological
  applications.
\newblock {\em Algebr. Geom. Topol.}, 1:311--320, 2001.

\bibitem{Tho03}
Simon Thomas.
\newblock The classification problem for torsion-free abelian groups of finite
  rank.
\newblock {\em J. Amer. Math. Soc.}, 16(1):233--258, 2003.

\bibitem{Tho15}
Simon Thomas.
\newblock A descriptive view of unitary group representations.
\newblock {\em J. Eur. Math. Soc.} 17, 1761--1787, 2015.

\end{thebibliography}
\end{document}